\documentclass[12pt,reqno]{amsart}
\usepackage{amssymb,latexsym}
\usepackage{amsmath}
\usepackage{amsthm}
\usepackage{graphicx}
\usepackage{hyperref}
\usepackage{titletoc}
\usepackage{pdfsync}
\usepackage{color}
\numberwithin{equation}{section}
\newtheorem{theorem}{Theorem}[section]

\newtheorem{lemma}[theorem]{Lemma}
\newtheorem{corollary}[theorem]{Corollary}

\usepackage{cite}

\theoremstyle{definition}

\def\R{{\mathfrak R}}

\def\begeq{\begin{equation}}
\def\endeq{\end{equation}}

\def\R{\Bbb R}

\allowdisplaybreaks

\begin{document}

\title[Some semilinear elliptic equations with Robin boundary condition]{
Uniqueness and nondegeneracy of positive solutions to  elliptic equations with Robin boundary conditions}

\author{Mengyao Chen}
\address{Mengyao Chen, School of Mathematics and Systems Science \& Hubei Province Key Laboratory of Systems Science in Metallurgical Process, Wuhan University of Science and Technology, Wuhan 430065, P. R. China.}
\email{cmy@mails.ccnu.edu.cn}

\author{Massimo Grossi}
\address{Massimo Grossi, Dipartimento di Matematica, Universit\`a di Roma ``La Sapienza'', P.le A. Moro 2, 00185 Roma, Italy}
\email{massimo.grossi@uniroma1.it}

\author{Qi Li}
\address{Qi Li, School of Mathematics and Systems Science \& Hubei Province Key Laboratory of Systems Science in Metallurgical Process, Wuhan University of Science and Technology, Wuhan 430065, P. R. China.}
\email{qili@mails.ccnu.edu.cn}

\begin{abstract}
In this paper, we study the uniqueness of positive solutions to the semilinear elliptic Robin problem
\begin{equation*}
\begin{cases}
-\Delta u = u^p, & \text{in } \Omega,\\
u > 0, & \text{in } \Omega,\\
\frac{\partial u}{\partial \nu} + \beta u = 0, & \text{on } \partial \Omega,
\end{cases}
\end{equation*}
where $\beta > 0$, $p$ is subcritical, and $\Omega$ is a bounded smooth domain. 
It is known that the uniqueness of the solution depends on the shape of the domain. Even if $\Omega$ is a ball, the problem is open for arbitrary $\beta>0$, since the method of moving planes does not work for Robin boundary conditions . By scaling arguments and a careful analysis of the linearized problem, we prove uniqueness for any $\beta>0$ provided that $p$ and $\Omega$ satisfy suitable conditions. Finally,  we study the effects of concave and convex nonlinearities.
\medskip


\medskip


\end{abstract}

\date{}
\maketitle

\section{Introduction and main results}
Let $\Omega$ be a bounded smooth domain in $\R^N$, with $N\ge2$, and let $\nu$ denote the unit outward normal to $\partial\Omega$. We study the uniqueness of positive solutions to the following problem 
\begin{equation}\label{eq0.1}
\begin{cases}
-\Delta u=u^p,\quad &\text{in}\ \Omega,\\
u>0, &\text{in}\ \Omega,\\
\frac{\partial u}{\partial \nu}+\beta u=0, &\text{on}\ \partial \Omega,
\end{cases}
\end{equation}
where $\beta>0$ is a parameter and $p$ is subcritical, namely, 
$$
1<p<\infty \quad \text{if}\ N=2,\quad 1<p<\frac{N+2}{N-2}\quad \text{if}\ N\ge 3.
$$
We work in the Sobolev space
\[
H^1(\Omega):=\{u\in L^2(\Omega): \nabla u\in L^2(\Omega)\},
\]
endowed with the norm
\[
\Vert u\Vert=\left(\int_{\Omega}\bigl(|\nabla u|^2+u^2\bigr)\,dx\right)^{1/2}.
\]
For each fixed $\beta>0$, we say that $u\in H^1(\Omega)$ is a weak solution of \eqref{eq0.1} if
\[
\int_{\Omega}\nabla u\nabla v\,dx+\beta\int_{\partial\Omega}uv\,dS
=\int_{\Omega}u^p v\,dx,
\qquad \forall v\in H^1(\Omega).
\]
In order to find a weak solution to \eqref{eq0.1}, it is enough to look for nontrivial critical points of the functional
\[
J_\beta(u)=\frac12\int_{\Omega}|\nabla u|^2\,dx+\frac{\beta}{2}\int_{\partial\Omega}u^2\,dS
-\frac{1}{p+1}\int_{\Omega}u_+^{p+1}\,dx.
\]
The corresponding solution $u_\beta$ is a weak solution and, by elliptic regularity, a classical one.

\subsection{Background}
Robin boundary conditions interpolate between the classical Dirichlet and Neumann cases, which formally correspond to the limiting regimes $\beta=\infty$ and $\beta=0$. There is a vast literature on these two problems. The existence of solutions and their qualitative properties, such as uniqueness, nondegeneracy, and asymptotic behavior, have been widely investigated over the last few decades. Given the breadth of the subject, we do not attempt a complete survey and refer the reader, for instance, to \cite{AY,BN,BO,GGP,GILY,GNN,GST,W,Z} and the references therein.

Compared with the Dirichlet and Neumann problems, the Robin case has been less extensively studied, although several existence and qualitative results are available in the literature. We refer, among others, to \cite{CGL,D1,D2,DF,W,vB,GS,F,BBBT} for results concerning eigenvalue estimates, existence, uniqueness, and asymptotic properties under Robin boundary conditions.

\subsection{Known results}
We first review the results for the case $p=0$, namely, 
\begin{equation}\label{1.2}
\begin{cases}
-\Delta u=1,\quad &\text{in}\ \Omega,\\
\frac{\partial u}{\partial \nu}+\beta u=0, &\text{on}\ \partial \Omega.
\end{cases}
\end{equation}
For each $\beta>0$, the positive solution is unique by the maximum principle. Assuming that $u_\beta$ is the unique solution of \eqref{1.2}, the estimate 
\begin{equation*}
\lambda_{1,\beta}^{-1}\leq \|u_\beta\|_{L^\infty(\Omega)}\leq
6N\lambda_{1,\beta}^{-1}\log\bigl(2^{11}3\sqrt{3}N(1+\beta^{-1}\sqrt{\lambda_{1,\beta}})\bigr)
\end{equation*}
was proved in \cite{vB}, where $\lambda_{1,\beta}$ is the first eigenvalue of the Robin eigenvalue problem.

The range $0<p<1$ corresponds to a sublinear nonlinearity and essentially exhibits the same phenomena as the case $p=0$. The authors in \cite{CGL} proved the uniqueness of positive solutions for each $\beta>0$. In addition, they established the exact leading-order asymptotics of the unique positive solution:
\[
u_\beta=
|\Omega|^{\frac{1}{1-p}}|\partial\Omega|^{-\frac{1}{1-p}}
\beta^{-\frac{1}{1-p}}
+o\bigl(\beta^{-\frac{1}{1-p}}\bigr),
\qquad \text{as }\beta\to0,
\]
where $|\Omega|$ and $|\partial\Omega|$ denote the volume of $\Omega$ and the surface measure of $\partial\Omega$, respectively.

When $p=1$, the problem \eqref{eq0.1} is closely related to the Robin eigenvalue problem 
\begin{equation}\label{1.3}
\begin{cases}
-\Delta u=\lambda_{1,\beta}u,\quad &\text{in }\Omega,\\
u>0,\quad &\text{in }\Omega,\\
\frac{\partial u}{\partial \nu}+\beta u=0,\quad &\text{on }\partial\Omega.
\end{cases}
\end{equation}
By the variational characterization 
$$
\lambda_{1,\beta}=\inf_{u\in H^1(\Omega)\backslash\{0\}}\frac{\int_{\Omega}|\nabla u|^2dx+\beta\int_{\partial\Omega}u^2dS}{\int_{\Omega}u^2dx},
$$
it is clear that $\lambda_{1,\beta}$ is non-decreasing with respect to $\beta$. One can easily prove that 
$$
\lim_{\beta\to0}\lambda_{1,\beta}=0\quad \lim_{\beta\to\infty}\lambda_{1,\beta}=\lambda_{1,\infty},
$$
where $\lambda_{1,\infty}$ is the first eigenvalue of the Dirichlet eigenvalue problem. It was proved in \cite{GS} that 
$$
\lim_{\beta\to0}\frac{\lambda_{1,\beta}}{\beta}=\frac{|\partial\Omega|}{|\Omega|}.
$$
If $\lambda_{i,\infty}$ is simple and $\phi_{i,\infty}$ denotes the unique (up to a sign) corresponding $L^2(\Omega)$-normalized eigenfunction, 
the first-order expansion of the eigenvalue 
$$
\lambda_{i,\beta}=\lambda_{i,\infty}-\frac{1}{\beta}\int_{\Omega}\Big(\frac{\partial \phi_{i,\infty}}{\partial\nu}\Big)^2dS+o\Big(\frac{1}{\beta}\Big),\quad \text{as}\ \beta\to\infty,
$$
was obtained by Filinovskiy in \cite{F}. Later, the second-order expansion of the eigenvalue was established in \cite{BBBT} through a functional-analytic approach. Recently, Ognibene \cite{O} developed a new approach, combining purely variational methods and singular perturbation theory, to prove the first-order expansion of any Robin eigenvalue under the sole Lipschitz regularity assumptions on the domains.

For subcritical exponents, one may consider the minimization problem 
\begin{equation}\label{1.4}
I_\beta=\inf_{u\in H^1(\Omega)\setminus\{0\}}
\frac{\int_{\Omega}|\nabla u|^2\,dx+\beta\int_{\partial\Omega}u^2\,dS}
{\left(\int_{\Omega}|u|^{p+1}\,dx\right)^{\frac{2}{p+1}}}.
\end{equation}
By arguments similar to those in \cite{Z}, one can verify that $I_\beta$ is achieved due to the compactness of the Sobolev embeddings, and a suitable scaling then yields a positive solution of \eqref{eq0.1}. 
Compared with the sublinear case, one may ask:\\
{\em Are positive solutions to equation \eqref{eq0.1} still unique for each $\beta>0$?}\\ 
In general, it is well known that uniqueness depends on the shape of the domain. But for the Robin problem, this remains open for arbitrary $\beta>0$, even if $\Omega$ is a ball. Indeed, the method of moving planes does not work for the Robin problem. Hence, in order to solve this uniqueness problem, we need a new approach to replace the classical moving planes.
In \cite{CGL}, the authors employed the scaling arguments of \cite{GS1} and proved the uniqueness of positive solutions for small $\beta$; similar results can also be found in \cite{DF,FD}. In addition, if $u_\beta$ is a positive solution, the exact leading-order asymptotic estimate 
\begin{equation}\label{1.5}
u_\beta=
|\Omega|^{-\frac{1}{p-1}}|\partial\Omega|^{\frac{1}{p-1}}
\beta^{\frac{1}{p-1}}
+o\bigl(\beta^{\frac{1}{p-1}}\bigr),
\qquad \text{as }\beta\to0.
\end{equation}
was established in \cite{CGL}.  

For the critical case, since the embedding $H^1(\Omega)\hookrightarrow L^{2^*}(\Omega)$ is no longer compact, the functional $J_\beta$ does not satisfy the Palais--Smale condition globally. In \cite{W}, Wang proved that $J_\beta$ satisfies the $(PS)_c$ condition provided
\[
0<c<\frac{1}{2N}\mathcal{S}^{N/2},
\]
where $\mathcal{S}$ denotes the best Sobolev constant for the embedding $D^{1,2}(\R^N)\hookrightarrow L^{2^*}(\R^N)$. As a consequence, one can prove that equation \eqref{eq0.1} admits a ground state solution if $\beta$ is small. In \cite{CGL}, the authors proved the uniqueness of the ground state solution for small $\beta$. In addition, they showed that the unique ground state solution has an asymptotic behavior similar to \eqref{1.5}.

\subsection{Main results}
According to the previous discussion, existence is known for the sublinear, subcritical, and critical cases. Uniqueness has been proved for the sublinear case, and for the subcritical case when $\beta$ is small. However, it remains open for arbitrary $\beta>0$.
The main purpose of this paper is to address the uniqueness of positive solutions for any $\beta>0$.

We now state our main results.

\begin{theorem}\label{theorem1.1}
Suppose that $\beta>0$ and $\Omega$ is a bounded smooth domain. Then there exists $p_0>1$ such that equation \eqref{eq0.1} has exactly one positive solution if $1<p<p_0$. Moreover, the unique positive solution is nondegenerate, in the sense that the linearized problem 
\begin{equation}\label{1.6}
\begin{cases}
-\Delta h-pu^{p-1}h=0, \quad &\text{in}\ \Omega,\\
\frac{\partial h}{\partial \nu}+\beta h=0 &\text{on}\ \partial\Omega,
\end{cases}
\end{equation}
has only the trivial solution $h\equiv 0$ in $H^1(\Omega)$.
\end{theorem}

As a direct consequence, we obtain the following corollary.
\begin{corollary}\label{corollary1.2}
Suppose that $\beta>0$ and $\Omega$ is a ball. Then the positive solution of equation \eqref{eq0.1} is unique and radial if $1<p<p_0$.
\end{corollary}

Moreover, one may ask:\\
{\em If $\Omega$ is a ball, are all positive radial solutions unique for all subcritical exponents $p$?}\\
We provide an answer in the following theorem. 

\begin{theorem}\label{theorem1.3}
Suppose that $\beta>0$, $p$ is subcritical, and $\Omega$ is a ball. If $N=2$, or $N\ge3$ with $\beta\ge N-2$, then each positive radial solution of equation \eqref{eq0.1} is nondegenerate, in the sense that the linearized problem \eqref{1.6} has only the trivial solution $h\equiv0$ in $H^1_{rad}(\Omega):=\{u\in H^1(\Omega) : u\ \text{is radial}\,\}$. Moreover, the positive radial solution of equation \eqref{eq0.1} is unique.
\end{theorem}

Next, assuming that $\Omega$ is bounded and convex, we consider the minimization  problem \eqref{1.4}. We obtain the following uniqueness result.
\begin{theorem}\label{Th1-4}
 Suppose that $p>1$ and $\Omega$ is bounded and convex in $\R^2$. Then there exists $\beta^*>0$ such that the minimizer of the problem \eqref{1.4} is unique for all $\beta>\beta^*$.
\end{theorem}

Finaly, as an application of Theorem \ref{theorem1.1}, we study the perturbed problem 
\begin{equation}\label{eq1.7}
\begin{cases}
-\Delta u=u^p+\gamma u^q,\quad &\text{in}\ \Omega,\\
u>0,  \quad &\text{in}\ \Omega,\\
\frac{\partial u}{\partial \nu}+\beta u=0, &\text{on}\ \partial\Omega,
\end{cases}
\end{equation}
where $\beta,\gamma>0$, $0<q<1$, and $p$ is subcritical. Then we have the following conclusion.

\begin{theorem}\label{theorem1.4}
Suppose that $\beta>0$ and $1<p<p_0$. Then there exists $\gamma_*>0$ such that 
equation \eqref{eq1.7} has exactly two positive solutions for $0<\gamma<\gamma_*$.

One of these solutions, denoted by $\bar u_\gamma$, is the minimal one and satisfies $\bar u_\gamma\to0$ as $\gamma\to0$. The other solution, denoted by $u_\gamma$, satisfies $u_\gamma\to u_0$, where $u_0>0$ is a solution of \eqref{m5}.
\end{theorem}

The paper is organized as follows. In Section 2, we collect some preliminary results and prove several auxiliary lemmas. In Section 3, we prove our uniqueness results. Finally, we present an application of our uniqueness results in Section 4.

\section{Preliminaries and asymptotic behavior as $p \to 1$}

In this section, we collect some preliminary results and establish several auxiliary lemmas. Let $u_p$ be a positive solution of \eqref{eq0.1}. Then we have the following result. 

\begin{lemma}\label{lemma2.1}
Suppose that $u_p$ is a solution of equation \eqref{eq0.1}, then we have 
$$
\|u_p\|_\infty^{p-1}\ge \lambda_{1,\beta}.
$$
\end{lemma}

\begin{proof}
Let $\phi_{1,\beta}>0$ be the first eigenfunction of the Robin eigenvalue problem 
\begin{equation}\label{eq1.3}
\begin{cases}
-\Delta \phi=\lambda_{1,\beta}\phi,\quad &\text{in}\ \Omega,\\
\frac{\partial \phi}{\partial \nu}+\beta \phi=0, &\text{on}\ \partial\Omega.
\end{cases}
\end{equation}
Multiplying equation \eqref{eq1.3} by $u_p$ and integrating by parts, we get
\begin{equation}\label{eq1.4}
\int_{\Omega}\nabla \phi_{1,\beta}\nabla u_p dx+\beta\int_{\partial\Omega}\phi_{1,\beta}u_p dS=\lambda_{1,\beta}\int_{\Omega}\phi_{1,\beta}u_p dx.
\end{equation}
On the other hand, multiplying equation \eqref{eq0.1} by $\phi_{1,\beta}$ and integrating by parts, we have 
\begin{equation}\label{eq1.5}
\int_{\Omega}\nabla u_p \nabla \phi_{1,\beta} dx+\beta\int_{\partial\Omega}u_p \phi_{1,\beta}dS=\int_{\Omega}u_p^p\phi_{1,\beta} dx.
\end{equation}
It follows from \eqref{eq1.4} and \eqref{eq1.5} that 
\begin{equation*}
\int_{\Omega}u_p\phi_{1,\beta}(u_p^{p-1}-\lambda_{1,\beta})dx=0,
\end{equation*}
which implies that 
$$
\max_{\Omega} (u_p^{p-1}-\lambda_{1,\beta})\ge 0.
$$
Thus, we can derive 
$$
\|u_p\|^{p-1}_\infty\ge \lambda_{1,\beta}.
$$
\end{proof}

In the following, we show that 
$$
\|u_p\|^{p-1}_\infty\to \lambda_{1,\beta},\quad \text{as}\ p\to1.
$$
To this end, we first prove that $\|u_p\|^{p-1}_\infty$ remains bounded as $p\to 1$.

\begin{lemma}\label{lemma2.4}
Suppose that $u_p$ is a solution of equation \eqref{eq0.1}, then there exists a constant $C_0$, independent of $p$, such that $\|u_p\|^{p-1}_\infty\le C_0$ as $p\to 1$.

\end{lemma}

\begin{proof}
Suppose that $\{p_n\}$ is a sequence such that $p_n>1$ and $p_n\to 1$ as $n\to\infty$. Let $u_n$ be a solution of \eqref{eq0.1} for $p=p_n$. We argue by contradiction. Suppose that 
$$
\|u_n\|^{p_n-1}_\infty\to \infty,\quad \text{as}\ n\to\infty.
$$
Define  
$$
\hat{u}_{n}(x)=\frac{1}{\|u_{n}\|_\infty}u_n\Big(\frac{x}{\|u_{n}\|_\infty^{\frac{p_n-1}{2}}}+x_{n}\Big),
$$
where $x_n$ is given by
$$
u_n(x_n)=\max_{x\in\Omega}u_n(x)=\|u_n\|_\infty.
$$
Then $\hat{u}_n$ is a solution of 
\begin{equation}
\begin{cases}
-\Delta \hat{u}_n=\hat{u}_n^{p_n},\quad &\text{in}\ \Omega_n,\\
\|u_n\|_\infty^{\frac{p_n-1}{2}}\frac{\partial \hat{u}_n}{\partial \nu}+\beta \hat{u}_n=0, &\text{on}\ \partial\Omega_n,
\end{cases}
\end{equation}
where $\Omega_n=\|u_n\|^{\frac{p_n-1}{2}}(\Omega-x_n)$. Thus, we distinguish two cases:

Case (i): If $\|u_n\|_\infty^{\frac{p_n-1}{2}}dist(x_n,\partial\Omega)\to \infty$ as $n\to\infty$, then $\Omega_n\to \R^N$. Since $0<\hat{u}_n\leq 1$, by elliptic regularity, we obtain that 
$$
\hat{u}_n\to \hat{u}_\infty\quad \text{in}\  C^2_{loc}(\R^N),
$$ 
where $\hat{u}_\infty$ is a positive solution of 
\begin{equation}\label{2.5}
\begin{cases}
-\Delta \hat{u}_\infty=\hat{u}_\infty,\quad &\text{in}\ \R^N,\\
\hat{u}_\infty>0 &\text{in}\ \R^N.
\end{cases}
\end{equation}
Let $\lambda_1^D$ and $\phi^D$ be respectively the first eigenvalue and the corresponding eigenfunction of the Dirichlet problem on $B_r$. Namely, 
\begin{equation}
\begin{cases}
-\Delta \phi^D=\lambda_1^D\phi^D,\quad &\text{in}\ B_r\\
\phi^D=0, &\text{on}\ \partial B_r.
\end{cases}
\end{equation}
A direct computation yields that 
$$
0>\int_{\partial B_r}\frac{\partial \phi^D}{\partial \nu}\hat{u}_\infty dS=(1-\lambda_1^D)\int_{B_r}\phi^D \hat{u}_\infty dx>0,
$$ 
provided $r$ is sufficiently large. This is a contradiction.

Case (ii): If $\|u_n\|_\infty^{\frac{p_n-1}{2}}dist(x_n,\partial\Omega)\to C\in[0,+\infty)$ as $n\to\infty$, then we have $x_n\to x_\infty\in\partial\Omega$. Since $\partial\Omega$ is smooth, after translation and rotation, we can assume that $x_\infty=0$ and that there exist $R>0$ and a smooth function $\varphi$ such that locally $\Omega$ is of the form 
\begin{align*}
\Omega\cap B_R&=\{ (x^\prime,x_N)\in B_R : x_N>\varphi(x^\prime)\},\\
\partial\Omega\cap B_R&=\{(x^\prime,x_N)\in B_R : x_N=\varphi(x^\prime)\},\\
\varphi(0)&=0,\quad \nabla \varphi(0)=0,
\end{align*}
where $B_R=\{x\in\R^N : |x|\le R\}$. Define a map $\Phi=(\Phi_1,\cdots,\Phi_N): B_R\to \R^N$ by
\begin{align*}
\Phi_i(x)&=x_i-\frac{\varphi(x^\prime)-x_N}{1+|\nabla \varphi(x^\prime)|^2}\frac{\partial \varphi(x^\prime)}{\partial x_i},\quad 1\le i\le N-1,\\
\Phi_N(x)&=x_N-\varphi(x^\prime).
\end{align*}
Then the determinant of the Jacobian of $\Phi$ at $0$ is clearly equal to $1$. Thus, by the inverse function theorem, there exist a constant $R_0<R$ and an open neighborhood of zero $\tilde{B}\subset B_R$ such that 
\begin{align*}
&\Phi: \Omega\cap\tilde{B}\to B_{R_0}^+=\{y\in B_{R_0} : y_N>0\},\\
&\Phi: \partial\Omega\cap \tilde{B}\to \{y\in B_{R_0} : y_N=0\}.
\end{align*}
Let 
$$
v_n(y)=u_n(\Phi^{-1}(y)),
$$
then $v_n$ solves the following problem 
\begin{equation*}
\begin{cases}
\displaystyle -\sum a_{ij}(y)\frac{\partial^2 v_n}{\partial y_i\partial y_j}+\sum b_j(y) \frac{\partial v_n}{\partial y_j}=v_n^{p_n},\quad &\text{in}\ B_{R_0}^+,\\
\displaystyle \alpha(y)\frac{\partial v_n}{\partial y_N}+\beta v_n(y)=0, &\text{on}\ B_{R_0}\cap\{y_N=0\},\\
v_n(y_n)=\|u_n\|_\infty,
\end{cases}
\end{equation*}
where $y_n$ is given by
$$
(y_{n,1},\cdots,y_{n, N})=y_n=\Phi(x_n),
$$
and $a_{ij}(y)$, $b_j(y)$, and $\alpha(y)$ are smooth functions satisfying 
$$
a_{ij}(y)=\delta_{ij}+O(|y|),\quad b_j(y)=(\Delta\Phi_{j})(\Phi^{-1}y),\quad \alpha(y)=-\sqrt{1+|\nabla \varphi(\Phi^{-1}(y^\prime,0))|^2}.
$$
By the definition of $\Phi$, we have 
\begin{equation}\label{2.7}
\lim_{n\to\infty}\|u_n\|_\infty^{\frac{p_n-1}{2}}dist(x_n,\partial\Omega)=\lim_{n\to\infty}\|u_n\|_\infty^{\frac{p_n-1}{2}}y_{n, N}=C.
\end{equation}
Arguing as in \cite[p.~891]{GS1}, one easily shows that the case $C=0$ cannot occur. Hence, $C>0$.

Define 
$$
\hat{v}_n (y)=\frac{1}{\|u_n\|_\infty}v_n\Big(\frac{y}{\|u_n\|_\infty^{\frac{p_n-1}{2}}}+y_n\Big),\quad y\in B_n=\|u_n\|_\infty^{\frac{p_n-1}{2}}(B_{R_0}^+-y_n),
$$
then we have $\hat{v}_n$ solves the following equation 
\begin{equation*}
\begin{cases}
\displaystyle -\sum a_{ij}\Big(\frac{y}{\|u_n\|_\infty^{\frac{p_n-1}{2}}}+y_n\Big)\frac{\partial^2 \hat{v}_n}{\partial y_i\partial y_j}+\frac{1}{\|u_n\|_\infty^{\frac{p_n-1}{2}}}\sum b_j\Big(\frac{y}{\|u_n\|_\infty^{\frac{p_n-1}{2}}}+y_n\Big) \frac{\partial \hat{v}_n}{\partial y_j}=\hat{v}_n^{p_n},\ \text{in}\ B_{n},\\
\displaystyle \|u_n\|_\infty^{\frac{p_n-1}{2}}\alpha\Big(\frac{y}{\|u_n\|_\infty^{\frac{p_n-1}{2}}}+y_n\Big)\frac{\partial \hat{v}_n}{\partial y_N}+\beta \hat{v}_n(y)=0, \ \text{on}\ B_n\cap \{y_N=-\|u_n\|_\infty^{\frac{p_n-1}{2}}y_{n, N}\},\\
\hat{v}_n(0)=1.
\end{cases}
\end{equation*}
We observe that $0<\hat{v}_n\le 1$. By \eqref{2.7} and elliptic regularity, we deduce that  
$$
\hat{v}_n\to \hat{v}_\infty\quad\text{in}\ C_{loc}^2(B_\infty),
$$
where $\hat{v}_\infty$ is a solution of 
\begin{equation*}
\begin{cases}
-\Delta \hat{v}_\infty=\hat{v}_\infty,\quad &\text{in}\ B_\infty,\\
\frac{\partial \hat{v}_\infty}{\partial y_N}=0, &\text{on}\ y_N=-C,\\
\hat{v}_\infty(0)=1,
\end{cases}
\end{equation*}
where $B_\infty=\{y\in\R^N : y_N>-C\}$.  
By performing an even reflection of $\hat{v}_\infty$ with respect to the hyperplane $y_N=-C$, we obtain a positive solution of \eqref{2.5}. Proceeding as in Case (i), we arrive at a contradiction. 
\end{proof}

As a direct consequence of Lemma \ref{lemma2.1} and \ref{lemma2.4}, we can conclude:

\begin{lemma}\label{lemma2.5}
Suppose that $u_p$ is a solution of equation \eqref{eq0.1}, then we have that 
$$
\lim_{p\to1} \|u_p\|_\infty^{p-1} = \lambda_{1,\beta}.
$$
\end{lemma}

\begin{proof}
By Lemma \ref{lemma2.1} and Lemma \ref{lemma2.4}, we have 
$$
\lambda_{1,\beta}\le \|u_p\|_\infty^{p-1}\le C_0.
$$
Thus, there exists a positive constant $\mu$ such that 
\begin{equation}\label{2.8}
\|u_p\|_\infty^{p-1}\to \mu,\quad \text{as}\ p\to1.
\end{equation}
Let 
$$
\tilde{u}_p=\frac{u_p}{\|u_p\|_\infty},
$$
then $\tilde{u}_p$ solves the following equation
\begin{equation}
\begin{cases}
-\Delta \tilde{u}_p=\|u_p\|_\infty^{p-1}\tilde{u}_p^p,\quad &\text{in}\ \Omega,\\
\frac{\partial \tilde{u}_p}{\partial \nu}+\beta\tilde{u}_p=0, &\text{on}\ \partial\Omega.
\end{cases}
\end{equation}
Since $0<\tilde{u}_p\le 1$, from \eqref{2.8} and elliptic regularity, we obtain 
$$
\tilde{u}_p\to \tilde{u}\quad \text{in}\ C^2(\bar{\Omega}),
$$
where $\tilde{u}$ is a positive solution of 
\begin{equation}
\begin{cases}
-\Delta \tilde{u}=\mu\tilde{u},\quad &\text{in}\ \Omega,\\
\frac{\partial \tilde{u}}{\partial \nu}+\beta\tilde{u}=0, &\text{on}\ \partial\Omega.
\end{cases}
\end{equation}
This implies that $\mu=\lambda_{1,\beta}$ and $\tilde{u}=\phi_{1,\beta}$.
\end{proof}

\section{Uniqueness and nondegeneracy of positive solutions}

{\bf Proof of Theorem \ref{theorem1.1}:}
We first prove uniqueness. We argue by contradiction. Suppose there exists a decreasing sequence $p_n \to 1$ such that $u_n := u_{p_n}$ and $v_n := v_{p_n}$ are two distinct positive solutions of \eqref{eq0.1} with $p=p_n$. 
Let 
$$
w_n=u_n-v_n,
$$
then $w_n$ is a solution of 
\begin{equation}
\begin{cases}
-\Delta w_n=a_n(x)w_n,\quad &\text{in}\ \Omega,\\
\frac{\partial w_n}{\partial \nu}+\beta w_n=0, &\text{on}\ \partial\Omega,
\end{cases}
\end{equation}
where $a_n(x)$ is defined by 
$$
a_n(x)=p_n\int_0^1(\theta u_n+(1-\theta)v_n )^{p_n-1}d\theta.
$$
Denoting $\tilde{u}_n := \tilde{u}_{p_n}$, by Lemma \ref{lemma2.5} we have $\tilde{u}_n \to \phi_{1,\beta}$ and $\|u_n\|_\infty^{p_n-1} \to \lambda_{1,\beta}$ as $n \to \infty$.
Thus, we have 
$$
u_n^{p_n-1}=\|u_n\|_\infty^{p_n-1}\tilde{u}_{n}^{p_n-1}\to \lambda_{1,\beta}\ \text{uniformly in } \Omega,\quad  \text{as}\ n\to\infty.
$$
Similarly, we have 
$$
v_n^{p_n-1}\to \lambda_{1,\beta}\ \text{uniformly in } \Omega,\quad \text{as}\ n\to\infty.
$$
Hence, we can deduce that 
$$
a_n(x)\to \lambda_{1,\beta}\ \text{uniformly in } \Omega, \quad \text{as}\ n\to\infty.
$$
Define 
$$
\tilde{w}_n=\frac{w_n}{\|w_n\|_\infty},
$$
then $\tilde{w}_n$ solves the following equation 
\begin{equation}
\begin{cases}
-\Delta \tilde{w}_n=a_n(x)\tilde{w}_n,\quad &\text{in}\ \Omega,\\
\frac{\partial \tilde{w}_n}{\partial \nu}+\beta\tilde{w}_n=0, &\text{on}\ \partial\Omega.
\end{cases}
\end{equation}
Since $|\tilde{w}_n|\leq 1$, by elliptic regularity, we deduce that $\tilde{w}_n \to \tilde{w}$ strongly in $C^2(\bar{\Omega})$, where $\tilde{w}$ is a solution of 
\begin{equation}
\begin{cases}
-\Delta \tilde{w}=\lambda_{1,\beta}\tilde{w},\quad &\text{in}\  \Omega,\\
\frac{\partial \tilde{w}}{\partial\nu}+\beta\tilde{w}=0, &\text{on}\ \partial\Omega.
\end{cases}
\end{equation}
This implies that $\tilde{w}=\phi_{1,\beta}$. A direct computation yields 
$$
0=\int_{\Omega}[(-\Delta u_n)v_n-(-\Delta v_n)u_n]dx=\int_{\Omega}u_nv_n(u_n^{p_n-1}-v_n^{p_n-1})dx,
$$
which implies that $w_n$ must change sign; otherwise, we would have $u_n \equiv v_n$.
We observe that  $\tilde{w}_n\to \phi_{1,\beta}$  in $C^2(\bar{\Omega})$ as $n\to\infty$. Since
\[
\frac{\partial \phi_{1,\beta}}{\partial \nu}<0
\quad\text{on }\partial\Omega,
\]
it follows that
\[
\frac{\partial \tilde{w}_n}{\partial \nu}<0
\quad\text{on }\partial\Omega
\]
for $n$ sufficiently large. This, in turn, implies that $\tilde{w}_n$ is positive in $\Omega$. This is a contradiction.

We now prove nondegeneracy. Suppose that $h_n$ is a solution of 
 \begin{equation}
\begin{cases}
-\Delta h_n-p_nu_n^{p_n-1}h_n=0, \quad &\text{in}\ \Omega,\\
\frac{\partial h_n}{\partial \nu}+\beta h_n=0 &\text{on}\ \partial\Omega,
\end{cases}
\end{equation}
and $\|h_n\|_{L^2(\Omega)}=1$.  
Since 
$$
\int_{\Omega}|\nabla u_n|^2dx+\beta\int_{\partial\Omega}u_n^2dS-p_n\int_{\Omega}u_n^{p_n-1}u_ndx=(1-p_n)\int_{\Omega}u_n^{p_n}dx<0,
$$
the first eigenvalue of the linearized problem is negative. Since $h_n$ is an eigenfunction corresponding to $0$, it must change sign.

On the other hand, since 
$$
u_n^{p_n-1}=\|u_n\|_\infty^{p_n-1}\tilde{u}_n^{p_n-1}\le \|u_n\|_\infty^{p_n-1}\le C_0,
$$
by elliptic regularity, we deduce that 
$$
h_n\to h \quad \text{in}\ C^2(\bar{\Omega})
$$
where $h$ is a solution of 
\begin{equation}
\begin{cases}
-\Delta h-\lambda_{1,\beta}h=0,\quad &\text{in}\ \Omega,\\
\frac{\partial h}{\partial \nu}+\beta h=0, &\text{on}\ \partial\Omega,
\end{cases}
\end{equation}
here we used the facts that 
$$
\tilde{u}_n\to \phi_{1,\beta},\quad \|u_n\|_\infty^{p_n-1}\to \lambda_{1,\beta}.
$$
Thus, we obtain $h=\phi_{1,\beta}$. Since $h_n \to \phi_{1,\beta}$ in $C^2(\bar{\Omega})$ as $n \to \infty$ and $\phi_{1,\beta}$ has a fixed sign, we deduce as before that $h_n$ must have a fixed sign for sufficiently large $n$. This is a contradiction since $h_n$ changes sign.
\qed

In the following, we will prove Theorem \ref{theorem1.3}. In order to prove this, we first prove the following conclusion, in which the proof is inspired by  \cite{DGP,L}.

\begin{lemma}\label{lemma3.1}
Suppose that $\beta>0$, $p$ is subcritical, and $\Omega$ is a ball. If $N=2$, or $N\ge 3$ with $\beta\ge N-2$, then each positive radial solution of \eqref{eq0.1} is nondegenerate in the space of radial functions.
\end{lemma}

\begin{proof}
Without loss of generality, we assume that $\Omega=B_1$. Let $u$ be a positive radial solution of \eqref{eq0.1}. It is sufficient to prove that \eqref{1.6} has only the trivial solution $h\equiv0$ in $H^1_{rad}(\Omega)$. It follows from \eqref{eq0.1} and \eqref{1.6} that 
$$
(p-1)\int_{\Omega}u^phdx=\int_{\Omega}[(-\Delta h)u-(-\Delta u)h]dx=0.
$$
Since $p>1$, we obtain 
\begin{equation}\label{eq3.6}
\int_{\Omega}u^phdx=0.
\end{equation}
If $h\equiv 0$, the proof is complete. Otherwise, \eqref{eq3.6} implies that $h$ must change sign. 

Define the nodal set of $h$ by
$$
\mathcal{M}=\overline{\{x\in\Omega : h(x)=0\}}.
$$
%
Since $h$ is radial, we claim that $\mathcal{M}\cap \partial\Omega=\emptyset$. Indeed, if there exists $x_0 \in \mathcal{M} \cap \partial\Omega$, we have $h(1)=0$. By the Robin boundary condition, this implies $h'(1)=-\beta h(1)=0$. Since $h$ satisfies the ODE
$$\begin{cases}
-h''-\frac{N-1}rh'=pu^{p-1}(r)h&\text{in }(0,1)\\
h(1)=h'(1)=0
\end{cases}
$$
uniqueness for the Cauchy problem immediately yields $h \equiv 0$, which is a contradiction.


Define 
$$
w=x\cdot\nabla u.
$$
Since $u$ is radial, then we get 
$$
w=x\cdot \nabla u=ru^\prime(r).
$$
A direct computation yields 
$$
\Delta w=2\Delta u+x\cdot \nabla (\Delta u)=-2u^p-pu^{p-1}w.
$$
This implies that $w$ satisfies 
\begin{equation}\label{3.7}
-\Delta w-pu^{p-1}w=2u^p.
\end{equation}
Multiplying \eqref{3.7} by $h$, we get
\begin{equation}\label{3.8}
\int_{\Omega}\nabla w\nabla hdx-\int_{\partial\Omega}\frac{\partial w}{\partial\nu}hdS-\int_{\Omega}pu^{p-1}whdx=2\int_{\Omega}u^phdx.
\end{equation}
On the other hand, multiplying \eqref{1.6} by $w$, we derive 
\begin{equation}\label{3.9}
\int_{\Omega}\nabla h\nabla wdx+\int_{\partial\Omega}\beta hwdS-\int_{\Omega}pu^{p-1}hwdx=0.
\end{equation}
It follows from \eqref{eq3.6}, \eqref{3.8} and \eqref{3.9} that 
\begin{equation}\label{3.10}
\int_{\partial\Omega}\Big(\frac{\partial w}{\partial\nu}+\beta w\Big)hdS=-2\int_{\Omega}u^phdx=0.
\end{equation}
Observe that 
$$
\nabla w=\nabla (ru^\prime(r))=(ru^\prime)^\prime\frac{x}{r}=(u^\prime+ru^{\prime\prime})\frac{x}{r},
$$
then for any $x\in\partial\Omega$ we have 
$$
\frac{\partial w}{\partial\nu}+\beta w=\nabla w\cdot x+\beta w=u^{\prime\prime}(1)+(\beta+1)u^\prime(1).
$$
Since $h$ has a fixed sign on $\partial\Omega$, by \eqref{3.10} we obtain  
\begin{equation*}
u^{\prime\prime}(1)+(\beta+1)u^\prime(1)=0.
\end{equation*}
Moreover, by the Robin boundary condition, we get $u'(1) = -\beta u(1)$, which yields 
\begin{equation}\label{eq3.11}
u^{\prime\prime}(1)=-(\beta+1)u^\prime(1)=\beta(\beta+1)u(1).
\end{equation}
On the other hand, it follows from \eqref{eq0.1} that 
$$
-u^{\prime\prime}(r)-\frac{N-1}{r}u^\prime=u^p,\quad 0<r\le 1.
$$
This implies that 
\begin{equation}\label{e3.12}
u^{\prime\prime}(1)=-u^p(1)-(N-1)u^\prime(1)=[(N-1)\beta-u^{p-1}(1)]u(1).
\end{equation}
Thus, it follows from \eqref{eq3.11} and \eqref{e3.12} that 
$$
u^{p-1}(1)=(N-2-\beta)\beta\le0.
$$
This is a contradiction since $u(1)>0$. 
\end{proof}

{\bf Proof of Theorem \ref{theorem1.3}:} It follows from Lemma \ref{lemma3.1} that 
each positive radial solution is nondegenerate. Then it is sufficient to prove uniqueness. In the following, we only consider the case $N\ge 3$, for $N=2$ the argument is the same. By Corollary \ref{corollary1.2}, there exists $p_0>1$ such that uniqueness holds for $1<p<p_0$. Let $(1,p_*)$ be the maximal interval where uniqueness holds. If $p_*=\frac{N+2}{N-2 }$, then the proof is complete. Otherwise, we may assume that $p_*<\frac{N+2}{N-2}$. Since all radial solutions of \eqref{eq0.1} are nondegenerate, we deduce by the implicit function theorem that uniqueness also holds for $p=p_*$.

We argue by contradiction. Suppose there exists a decreasing sequence $\{p_n\}$ with $p_n<\frac{N+2}{N-2}$ such that $p_n\to p_*$, and equation \eqref{eq0.1} with $p=p_n$ has two distinct radial solutions $u_{p_n}$ and $v_{p_n}$. To simplify notation, we write $u_n=u_{p_n}$ and $v_n=v_{p_n}$. We first claim that there exists a constant $C_1>0$ such that $\|u_n\|_\infty \le C_1$.
If not, assuming $\|u_n\|_\infty\to\infty$ and proceeding as in the proof of Lemma \ref{lemma2.4}, we deduce that $\hat{u}_n\to\hat{u}_\infty$ in $C^2_{loc}(\Omega_\infty)$, where $\hat{u}_\infty$ is a positive solution of 
\begin{equation}\label{eq3.12}
\begin{cases}
-\Delta \hat{u}_\infty=\hat{u}_\infty^{p_*},\quad &\text{in}\ \Omega_\infty,\\
0\le\hat{u}_\infty\le 1 &\text{in}\ \Omega_\infty,\\
\frac{\hat{u}_\infty}{\partial x_N}=0, &\text{on}\ \partial\Omega_\infty \ (\text{if}\ \Omega_\infty=\R^N_+),
\end{cases}
\end{equation}
where $\Omega_\infty=\R^N$ or $\Omega_\infty=\R^N_+$. Moreover, we get that $\hat{u}_\infty\equiv 0$, which is impossible since $\hat{u}_\infty(0)=1$. Thus, the claim holds. Similarly, we obtain that $\|v_n\|_\infty$ is bounded. By elliptic regularity, both $u_n$ and $v_n$ converge to $u_{p_*}$ in $C^2(\bar{\Omega})$, since \eqref{eq0.1} has a unique radial solution $u_{p_*}$ for $p=p_*$. Let 
$$
\tilde{w}_n=\frac{u_n-v_n}{\|u_n-v_n\|_\infty},
$$
then $\tilde{w}_n$ solves the equation 
\begin{equation}
\begin{cases}
-\Delta \tilde{w}_n=a_n(x)\tilde{w}_n,\quad &\text{in}\ \Omega,\\
\frac{\partial \tilde{w}_n}{\partial \nu}+\beta \tilde{w}_n=0, &\text{on}\ \partial\Omega,
\end{cases}
\end{equation}
where $a_n(x)$ is defined by
$$
a_n(x)=p_n\int_0^1(\theta u_n+(1-\theta)v_n)^{p_n-1}d\theta.
$$
Since $a_n \to p_* u_{p_*}^{p_*-1}$ and $|\tilde{w}_n| \le 1$, by elliptic regularity we obtain that $\tilde{w}_n \to \tilde{w}_*$ strongly in $C^2(\bar{\Omega})$, where $\tilde{w}_*$ is a solution of 
\begin{equation}
\begin{cases}
-\Delta \tilde{w}_*=p_*u_{p_*}^{p_*-1}\tilde{w}_*,\quad &\text{in}\  \Omega,\\
\frac{\partial \tilde{w}_*}{\partial\nu}+\beta\tilde{w}_*=0, &\text{on}\ \partial\Omega.
\end{cases}
\end{equation}
By nondegeneracy, we obtain $\tilde{w}_*\equiv 0$. This is a contradiction since $\|\tilde{w}_*\|_\infty=1$.
\qed

Finally, we study the bounded and convex domains. Suppose that $p$ is subcritical, but $\beta$ is not fixed. Thus, the positive solutions of \eqref{eq0.1} depend on $\beta$. Since positive solutions may blow up near the boundary as $\beta\to\infty$, we turn to study the minimization problem \eqref{1.4} and prove the uniqueness of the minimizer.

{\bf Proof of Theorem \ref{Th1-4}:} One easily sees that the minimization problem \eqref{1.4} is equivalent to 
\begin{equation}\label{e3.16}
I_\beta=\inf \Big\{ \int_{\Omega}|\nabla u|^2dx+\beta\int_{\partial\Omega}u^2dS :  u\in H^1(\Omega), \ \int_{\Omega}|u|^{p+1}dx=1 \Big\}.
\end{equation}
Obviously, $I_\beta>0$. Define 
$$
I_\infty=\inf \Big\{ \int_{\Omega}|\nabla u|^2dx :  u\in H_0^1(\Omega), \ \int_{\Omega}|u|^{p+1}dx=1 \Big\}.
$$
Since $H_0^1(\Omega)\subset H^1(\Omega)$, we have 
$$
I_\beta\le I_\infty<\infty.
$$
We argue by contradiction. Suppose there exists an increasing sequence $\beta_n\to\infty$ such that the minimization  problem \eqref{e3.16} with $\beta=\beta_n$ has two distinct minimizers $u_n$ and $v_n$. By a direct computation, we obtain 
$$
\lambda_{1,\beta_n}\int_{\Omega}u_n^2dx\le\int_{\Omega}|\nabla u_n|^2dx+\beta_n\int_{\partial\Omega}u_n^2dS=I_{\beta_n}.
$$
We observe that $\lambda_{1,\beta_n}\to \lambda_{1,\infty}>0$ as $n\to\infty$, where $\lambda_{1,\infty}$ is the first eigenvalue of the Dirichlet problem. Thus, we have that $u_n$ is bounded in $H^1(\Omega)$. Moreover, there exists $u_\infty\in H^1(\Omega)$ such that 
\begin{align*}
u_n\rightharpoonup u_\infty,\quad &\text{weakly in}\ H^1(\Omega),\\
u_n\to u_\infty,\quad &\text{strongly in}\ L^{p+1}(\Omega),\\
u_n\to u_\infty,\quad &\text{strongly in}\ L^{2}(\partial\Omega).
\end{align*}
Since 
$$
\int_{\partial\Omega}u_n^2dS\le \frac{I_{\beta_n}}{\beta_n}\to 0,\quad \text{as}\ n\to\infty,
$$
we obtain $u_\infty=0$ on $\partial\Omega$. Furthermore, since 
$$
\int_{\Omega}|\nabla u_\infty|^2dx\le \underset{n\to\infty} {\lim\inf}\int_{\Omega}|\nabla u_n|^2dx\le \lim_{n\to\infty} I_{\beta_n}\le  I_\infty,
$$
and
$$
\int_{\Omega}|u_\infty|^{p+1}dx=\lim_{n\to\infty}\int_{\Omega}|u_n|^{p+1}dx=1,
$$
we find that $u_\infty$ is a minimizer of $I_\infty$. Similarly, we obtain $v_n \rightharpoonup v_\infty$ weakly in $H^1(\Omega)$, and $v_\infty$ is also a minimizer of $I_\infty$. From \cite{L}, we know that $u_\infty=v_\infty$ and that it is nondegenerate. Let 
$$
w_n=u_n-v_n,\quad \tilde{w}_n=\frac{w_n}{\|w_n\|_{H^1(\Omega)}},
$$
then we have that $\tilde{w}_n$ is a solution of 
\begin{equation*}
\begin{cases}
-\Delta \tilde{w}_n=\tilde{a}_n(x)\tilde{w}_n,\quad &\text{in}\  \Omega,\\
\frac{\partial \tilde{w}_n}{\partial \nu}+\beta_n \tilde{w}_n=0, &\text{on}\ \partial\Omega,
\end{cases}
\end{equation*}
where $\tilde{a}_n(x)$ is defined by 
$$
\tilde{a}_n(x)=p\int_0^1(\theta u_n+(1-\theta)v_n)^{p-1}d\theta.
$$
Since $\tilde{w}_n$ is bounded in $H^1(\Omega)$, there exists $\tilde{w}_\infty\in H^1(\Omega)$ such that 
\begin{align*}
\tilde{w}_n\rightharpoonup \tilde{w}_\infty,\quad &\text{weakly in}\ H^1(\Omega),\\
\tilde{w}_n\to \tilde{w}_\infty,\quad &\text{strongly in}\ L^{p+1}(\Omega),\\
\tilde{w}_n\to \tilde{w}_\infty,\quad &\text{strongly in}\ L^{2}(\partial\Omega).
\end{align*}
By direct computation, we obtain 
$$
\int_{\Omega}|\nabla \tilde{w}_n|^2dx+\beta_n\int_{\partial\Omega}\tilde{w}_n^2dS=\int_{\Omega}\tilde{a}_n(x)\tilde{w}_n^2dx=\int_{\Omega}pu_\infty^{p-1}\tilde{w}_\infty^2dx+o(1).
$$
By the properties of the first eigenvalue of the Robin problem, we have 
$$
\int_{\Omega}|\nabla \tilde{w}_n|^2dx+\beta_n\int_{\partial\Omega}\tilde{w}_n^2dS\ge \min\Big\{\frac{1}{2},\frac{\lambda_{1,\beta_n}}{2}\Big\}\|\tilde{w}_n\|_{H^1(\Omega)}^2=\min\Big\{\frac{1}{2},\frac{\lambda_{1,\infty}+o(1)}{2}\Big\}.
$$
Thus, we deduce that $\tilde{w}_\infty\neq 0$. We observe that 
$$
\int_{\partial\Omega}\tilde{w}_n^2dS\le\frac{1}{\beta_n}\int_{\Omega}pu_\infty^{p-1}\tilde{w}_\infty^2dx+o\Big(\frac{1}{\beta_n}\Big)\to 0,\quad \text{as}\ n\to\infty,
$$
which implies $\tilde{w}_\infty=0$ on $\partial \Omega$. For any $\phi\in H_0^1(\Omega)$, we have 
$$
\int_{\Omega}\nabla \tilde{w}_n\nabla\phi dx=\int_{\Omega}\tilde{a}_n(x)\tilde{w}_n\phi dx.
$$
Passing to the limit, we get 
$$
\int_{\Omega}\nabla \tilde{w}_\infty\nabla\phi dx=\int_{\Omega}pu_\infty^{p-1}\tilde{w}_\infty\phi dx.
$$
This implies that $\tilde{w}_\infty$ is a nontrivial solution of 
\begin{equation*}
\begin{cases}
-\Delta \tilde{w}_\infty=pu_\infty^{p-1}\tilde{w}_\infty,\quad &\text{in}\ \Omega,\\
\tilde{w}_\infty=0, &\text{on}\ \partial\Omega,
\end{cases}
\end{equation*}
which is a contradiction since $u_\infty$ is nondegenerate.
\qed

\section{Multiplicity of solutions for a perturbed Robin problem}

Let $\tilde{J}_\gamma$ be the functional of equation \eqref{eq1.7}, that is 
$$
\tilde{J}_\gamma (u)=\frac{1}{2}\int_{\Omega}|\nabla u|^2dx+\frac{\beta}{2}\int_{\partial\Omega}u^2dS-\frac{1}{p+1}\int_\Omega u_+^{p+1}dx-\frac{\gamma}{q+1}\int_{\Omega}u_+^{q+1}dx.
$$
From the variational structure, one can show that equation \eqref{eq1.7} has a first solution which is a local minimizer of the functional, and a second solution can be found by means of the Mountain Pass Theorem (see, e.g., \cite{ABC}). To show that equation \eqref{eq1.7} has exactly two solutions, we first establish several preliminary lemmas.

\begin{lemma}\label{lemma4.1}
Suppose that $0<q<1$ and $\beta>0$. Then:
\begin{enumerate}
\item The positive solution of 
\begin{equation}\label{l4.1}
\begin{cases}
-\Delta u=u^q,\quad &\text{in}\ \Omega,\\
u>0, &\text{in}\ \Omega,\\
\frac{\partial u}{\partial \nu}+\beta u=0, &\text{on}\ \partial \Omega,
\end{cases}
\end{equation}
is unique and nondegenerate; 
\item  If $w$ and $v$ satisfy 
\begin{equation*}
\begin{cases}
-\Delta w\ge w^q,\quad &\text{in}\ \Omega,\\
w>0, &\text{in}\ \Omega,\\
\frac{\partial w}{\partial \nu}+\beta w=0, &\text{on}\ \partial \Omega,
\end{cases}
\end{equation*}
and
\begin{equation*}
\begin{cases}
-\Delta v\le v^q,\quad &\text{in}\ \Omega,\\
v>0, &\text{in}\ \Omega,\\
\frac{\partial v}{\partial \nu}+\beta v=0, &\text{on}\ \partial \Omega,
\end{cases}
\end{equation*}
then $w\ge v$.
\end{enumerate}
\end{lemma}

\begin{proof}
We first prove assertion (1). 
Uniqueness was proved in \cite{CGL}. We now prove nondegeneracy. Let $u$ be the unique positive solution. It is sufficient to show that $\tilde{\lambda}_{1,\beta}>0$, where  $\tilde{\lambda}_{1,\beta}$ is the first eigenvalue of the linearized operator $-\Delta -qu^{q-1}$ with Robin boundary conditions.

Note that $u$ has a lower bound $c_\beta>0$, which depends on $\beta$. Moreover, we get $u^{q-1}\le c_\beta^{q-1}$. If we fixed $\beta>0$, then $c_\beta$ is a fixed constant. So $u^{q-1}\le c_\beta^{q-1}<\infty$ on the boundary.

One easily sees that the minimization problem 
$$
\min_{\psi\in H^1(\Omega)}\Big\{\frac{1}{2}\int_{\Omega}|\nabla \psi|^2dx+\frac{\beta}{2}\int_{\partial\Omega}\psi^2dS-\frac{1}{q+1}\int_{\Omega}|\psi|^{q+1}dx \Big\}
$$
is achieved. Moreover, $u$ is a minimizer. Thus, we have 
$$
\int_{\Omega}|\nabla h|^2dx+\beta\int_{\partial\Omega}h^2dS-q\int_{\Omega}u^{q-1}h^2\geq 0,\quad \forall h\in H^1(\Omega),
$$
which implies $\tilde{\lambda}_{1,\beta}\ge 0$. If $\tilde{\lambda}_{1,\beta}=0$, letting $\tilde{\phi}_{1,\beta}>0$ be the corresponding first eigenfunction, we have 
\begin{equation}\label{eq4.1}
\begin{cases}
-\Delta \tilde{\phi}_{1,\beta}-qu^{q-1}\tilde{\phi}_{1,\beta}=0,\quad &\text{in}\ \Omega,\\
\frac{\partial \tilde{\phi}_{1,\beta}}{\partial \nu}+\beta \tilde{\phi}_{1,\beta}=0 &\text{on}\ \partial\Omega.
\end{cases}
\end{equation}
Multiplying \eqref{eq4.1} by $u$ and integrating by parts, we obtain
\begin{equation}\label{eq4.2}
\int_{\Omega}\nabla \tilde{\phi}_{1,\beta}\nabla udx+\beta\int_{\partial\Omega}\tilde{\phi}_{1,\beta}udS-q\int_{\Omega}u^q\tilde{\phi}_{1,\beta}dx=0.
\end{equation}
Similarly, from \eqref{l4.1} we derive 
\begin{equation}\label{eq4.3}
\int_{\Omega}\nabla u \nabla \tilde{\phi}_{1,\beta}dx+\beta\int_{\partial\Omega}u\tilde{\phi}_{1,\beta}dS-\int_{\Omega}u^q\tilde{\phi}_{1,\beta}dx=0.
\end{equation}
It follows from these two identities that 
$$
(1-q)\int_{\Omega}u^q\tilde{\phi}_{1,\beta}dx=0,
$$
which implies 
$$
\int_{\Omega}u^q\tilde{\phi}_{1,\beta}dx=0.
$$
This is a contradiction since both $u$ and $\tilde{\phi}_{1,\beta}$ are positive. Hence, we obtain $\tilde{\lambda}_{1,\beta}>0$. 

Next, we prove assertion (2). Inspired by \cite{BK}, let $\xi(t)$ be a smooth, nondecreasing function such that $\xi(0)=0$, $\xi(t)=1$ for $t\ge1$, and $\xi(t)=0$ for $t\le 0$. Define 
$$
\xi_\varepsilon(t)=\xi\Big(\frac{t}{\varepsilon} \Big),
$$
then we get 
\begin{equation}\label{4.4}
\int_{\Omega}(-v\Delta w+w\Delta v)\xi_\varepsilon(v-w)dx\ge \int_{\Omega}wv(w^{q-1}-v^{q-1})\xi_\varepsilon(v-w)dx.
\end{equation}
A direct computation yields
\begin{align*}
\text{left hand side of \eqref{4.4}}&=\int_{\Omega}\nabla w\nabla (v\xi_\varepsilon(v-w))dx+\beta\int_{\partial\Omega}wv\xi_\varepsilon(v-w)dS\\
&\quad-\int_{\Omega}\nabla v\nabla (w\xi_\varepsilon(v-w))dx-\beta\int_{\partial\Omega}vw\xi_\varepsilon(v-w)dS\\
&=\int_{\Omega}\xi_\varepsilon(v-w)\nabla w\nabla vdx+\int_{\Omega}v\xi_\varepsilon^{\prime}(v-w)\nabla w(\nabla v-\nabla w)dx\\
&\quad-\int_{\Omega}\xi_\varepsilon(v-w)\nabla v\nabla wdx-\int_{\Omega}w\xi_\varepsilon^\prime(v-w)\nabla v(\nabla v-\nabla w)dx\\
&=\int_{\Omega}\xi_{\varepsilon}^{\prime}(v-w)(\nabla v-\nabla w)(v\nabla w-w\nabla v+v\nabla v-v\nabla v)dx\\
&=\int_{\Omega}\xi_{\varepsilon}^{\prime}(v-w)[-v|\nabla v-\nabla w|^2+(v-w)\nabla v(\nabla v-\nabla w)]dx\\
&\le\int_{\Omega}(v-w)\xi_{\varepsilon}^{\prime}(v-w)\nabla v(\nabla v-\nabla w)dx\\
&=\int_{\Omega}\nabla v\nabla [\eta_\varepsilon(v-w)]dx\\
&=\int_{\Omega}-\eta_\varepsilon(v-w)\Delta v-\beta\int_{\partial\Omega}v\eta_\varepsilon(v-w)dS,
\end{align*}
where $\eta_\varepsilon(t)=\int_0^ts\xi^\prime_\varepsilon(s) \, ds$. Since 
$$
0\le \eta_\varepsilon(t)\le \int_0^\varepsilon s\xi^\prime_\varepsilon(s)ds=\varepsilon-\int_0^\varepsilon\xi_\varepsilon(s)ds\le \varepsilon,
$$
by \eqref{4.4} we obtain 
$$
\int_{\Omega}wv(w^{q-1}-v^{q-1})\xi_\varepsilon(v-w)dx\le C\varepsilon,
$$
for some constant $C>0$. Letting $\varepsilon\to0$, we obtain 
$$
\int_{\text{meas}\{v>w\}}wv(w^{q-1}-v^{q-1})dx\le 0,
$$
which implies that $\text{meas}\{v>w\}=0$, since $s^{q-1}$ is strictly decreasing on $(0,\infty)$. Therefore, we conclude that $w\ge v$.
\end{proof}

Define 
$$
\gamma_0:=\sup\{\gamma>0: \eqref{eq1.7}\ \text{admits a solution} \}.
$$
Then we have the following conclusion.
\begin{lemma}\label{lemma4.2}
Suppose that $\beta>0$, $0<q<1$, and $p$ is subcritical. Then:
\begin{enumerate}
\item $0<\gamma_0<\infty$;
\item For all $0<\gamma<\gamma_0$, equation \eqref{eq1.7} has a minimal solution $\bar{u}_\gamma$, which is strictly increasing with respect to $\gamma$.
\end{enumerate}
\end{lemma}

\begin{proof}
We first prove assertion (1). Let $u$ be the unique solution of \eqref{1.2}. Defining
\begin{equation}\label{m2}
M=(2\|u\|_\infty^q\gamma)^{\frac{1}{1-q}},
\end{equation}
we have 
$$
M-M^p\|u\|_\infty^p-\gamma M^q\|u\|_\infty^q=\frac{M}{2}-M^p\|u\|_\infty^p.
$$
Since $0<q<1<p$, we can find $\gamma_1>0$ such that for all $0<\gamma<\gamma_1$,
$$
M-M^p\|u\|_\infty^p-\gamma M^q\|u\|_\infty^q\ge0.
$$
Then we have 
$$
-\Delta (Mu)=M\ge (Mu)^p+\gamma (Mu)^q,
$$
which implies that $Mu$ is a supersolution of \eqref{eq1.7}. On the other hand, by \eqref{eq1.3} we get 
$$
-\Delta (\varepsilon\phi_{1,\beta})=\varepsilon\lambda_{1,\beta} \phi_{1,\beta}\le (\varepsilon \phi_{1,\beta})^p+\gamma(\varepsilon \phi_{1,\beta})^q
$$
for all sufficiently small $\varepsilon>0$. This implies that $\varepsilon\phi_{1,\beta}$ is a subsolution of \eqref{eq1.7}. Choosing $\varepsilon$ sufficiently small, we have 
$$
\varepsilon\phi_{1,\beta}<Mu.
$$ 
Thus, by the method of subsolutions and supersolutions (see, e.g., \cite{I,P}), \eqref{eq1.7} has a solution for all $\gamma\le \gamma_1$, which implies $\gamma_0\ge \gamma_1$. Since the function $\lambda_{1,\beta}s^{1-q}-s^{p-q}$ has a maximum on $[0, \infty)$, we define 
\begin{equation}\label{eq4.4}
\gamma_2=\max_{s\ge 0}\{ \lambda_{1,\beta}s^{1-q}-s^{p-q}\}.
\end{equation}
If $u_\gamma$ is a solution of \eqref{eq1.7}, we derive 
\begin{align*}
\lambda_{1,\beta}\int_{\Omega}u_\gamma\phi_{1,\beta}dx=\int_{\Omega}(-\Delta \phi_{1,\beta})u_\gamma dx=\int_{\Omega}(-\Delta u_\gamma)\phi_{1,\beta}dx=\int_{\Omega}u_\gamma^p\phi_{1,\beta}dx+\gamma\int_{\Omega}u_\gamma^q\phi_{1,\beta}dx.
\end{align*}
Moreover, from \eqref{eq4.4} we obtain 
$$
\gamma\int_{\Omega}u_\gamma^q\phi_{1,\beta}dx=\int_{\Omega}u_\gamma^q\phi_{1,\beta}(\lambda_{1,\beta}u_\gamma^{1-q}-u_\gamma^{p-q})dx\le\gamma_2\int_{\Omega}u_\gamma^q\phi_{1,\beta}dx.
$$
Hence, $\gamma\le \gamma_2$. Since $\gamma_2$ is independent of $\gamma$, we obtain $\gamma_0\le\gamma_2$.

Next, we prove assertion (2). For each $\gamma<\gamma_0$, by the definition of $\gamma_0$, there exists $\alpha\in(\gamma,\gamma_0)$ such that 
\begin{equation*}
\begin{cases}
-\Delta u=u^p+\alpha u^q,\quad &\text{in}\ \Omega,\\
u>0,  \quad &\text{in}\ \Omega,\\
\frac{\partial u}{\partial \nu}+\beta u=0, &\text{on}\ \partial\Omega,
\end{cases}
\end{equation*}
has a solution $u_\alpha$. Then $u_\alpha$ is a supersolution of \eqref{eq1.7}. It follows from assertion (1) that $\varepsilon\phi_{1,\beta}<u_\alpha$ is a subsolution for sufficiently small $\varepsilon>0$. Thus, by the method of subsolutions and supersolutions, \eqref{eq1.7} has a solution $u_\gamma$. Let $v_\gamma$ be the unique positive solution of 
\begin{equation}\label{cgl4.8}
\begin{cases}
-\Delta v=\gamma v^q,\quad &\text{in}\ \Omega,\\
v>0, &\text{in}\ \Omega,\\
\frac{\partial v}{\partial \nu}+\beta v=0, &\text{on}\ \partial \Omega,
\end{cases}
\end{equation}  
since 
$$
-\Delta u_\gamma=u_\gamma^p+\gamma u_\gamma^q>\gamma u_\gamma^q,
$$
it follows from Lemma \ref{lemma4.1} that $u_\gamma\ge v_\gamma$. We observe that $v_\gamma$ is a subsolution of \eqref{eq1.7}. 

We denote by $\langle v_\gamma, u_\gamma \rangle$ the order interval consisting of all functions
$u\in C(\overline{\Omega})$ satisfying
\[
v_\gamma\le u\le u_\gamma.
\]
Let $f(u)=u^p+\gamma u^q$. Since $f$ is nondecreasing on
$\langle v_\gamma, u_\gamma \rangle$, we can argue as in \cite{P} and construct a sequence
$\{u_k\}$ by the monotone iteration scheme
\begin{equation*}
\begin{cases}
-\Delta u_{k+1}=f(u_k), & \text{in }\Omega,\\
\dfrac{\partial u_{k+1}}{\partial \nu}+\beta u_{k+1}=0,
& \text{on }\partial\Omega,\\
u_0=v_\gamma.
\end{cases}
\end{equation*}
The sequence satisfies
\[
v_\gamma\le \cdots \le u_k\le u_{k+1}\le \cdots \le u_\gamma.
\]
Hence, $\{u_k\}$ converges increasingly to a limit $\bar u_\gamma$, which is the minimal solution of \eqref{eq1.7}. Therefore, problem \eqref{eq1.7} admits a minimal solution $\bar u_\gamma$ for every $\gamma<\gamma_0$.

Furthermore, for any $0<\gamma<\gamma^\prime<\gamma_0$, we have 
$$
-\Delta \bar{u}_{\gamma^\prime}=\bar{u}_{\gamma^\prime}^p+\gamma^\prime \bar{u}_{\gamma^\prime}>\bar{u}_{\gamma^\prime}^p+\gamma \bar{u}_{\gamma^\prime},
$$
which implies that $\bar{u}_{\gamma^\prime}$ is a supersolution of \eqref{eq1.7}. Since $\varepsilon\phi_{1,\beta}$ is a subsolution for sufficiently small $\varepsilon>0$, the method of subsolutions and supersolutions yields a solution $u_\gamma$ of \eqref{eq1.7} satisfying $\varepsilon\phi_{1,\beta}\le u_\gamma\le \bar{u}_{\gamma^\prime}$. Since $\bar{u}_\gamma$ is the minimal solution, we obtain $\bar{u}_\gamma\le u_\gamma\le \bar{u}_{\gamma^\prime}$, which shows that $\bar{u}_\gamma$ is strictly increasing with respect to $\gamma$.
\end{proof}

\begin{lemma}\label{lemma4.3}
For all $0<\gamma<\gamma_0$, equation \eqref{eq1.7} has at most one solution $u_\gamma$ satisfying
\begin{equation}\label{m3}
\|u_\gamma\|_\infty<\Big(\frac{\tilde{\lambda}_{1,\beta}}{p}\Big)^{\frac{1}{p-1}},
\end{equation}
where $\tilde{\lambda}_{1,\beta}$ is defined in Lemma \ref{lemma4.1}. Moreover, equation \eqref{eq1.7} has only the minimal solution $\bar{u}_\gamma$ when $\gamma$ is sufficiently small. 
\end{lemma}

\begin{proof}
The proof is standard (see, e.g., \cite{ABC}); for the convenience of the reader, we provide a brief sketch. Suppose that equation \eqref{eq1.7} has two distinct solutions. 
By Lemma \ref{lemma4.2}, one of these solutions must be the minimal solution $\bar{u}_\gamma$. We write the second solution as $u_\gamma:=\bar{u}_\gamma+v$, where the strong maximum principle implies $v>0$. It follows from Lemma \ref{lemma4.1} that  $\bar{u}_\gamma\ge v_\gamma$, where $v_\gamma$ is the unique positive solution of \eqref{cgl4.8} Thus, by \eqref{m3} we obtain
\begin{equation}\label{eq4.7}
-\Delta v- \gamma qv_\gamma^{q-1}v\le -\Delta v-\gamma q\bar{u}_\gamma^{q-1}v\le (\bar{u}_\gamma+v)^p-\bar{u}_\gamma^p\le p(\bar{u}_\gamma+v)^{p-1}v<\tilde{\lambda}_{1,\beta}v,
\end{equation}
where we use the following facts
$$
(a+b)^q-a^q\le qa^{q-1}b, \quad (a+b)^p-b^p\le p(a+b)^{p-1},\quad \forall a,b>0,
$$
since $s^q$ is concave and $s^p$ is convex.
Let $u=\gamma^{\frac{1}{q-1}}v_\gamma$. Then $u$ is a solution of \eqref{l4.1}. Moreover, from \eqref{eq4.7} we obtain 
$$
-\Delta v-qu^{q-1}v=-\Delta v- \gamma qv_\gamma^{q-1}v<\tilde{\lambda}_{1,\beta}v.
$$
This implies 
$$
\frac{\int_{\Omega}|\nabla v|^2dx+\beta\int_{\partial \Omega}v^2dS-q\int_{\Omega}u^{q-1}v^2dx}{\int_{\Omega}v^2dx}<\tilde{\lambda}_{1,\beta},
$$
which contradicts the definition of $\tilde{\lambda}_{1,\beta}$.

From  \eqref{m2}, we observe that $M\to 0$ as $\gamma\to 0$. Then, by the method of subsolutions and supersolutions, we obtain $\|\bar{u}_\gamma\|_\infty\to 0$ as $\gamma\to 0$. Hence, $\bar{u}_\gamma$ is the unique solution when $\gamma$ is sufficiently small. 
\end{proof}

\begin{lemma}\label{lemma4.4}
Suppose that $\beta>0$, $0<q<1$, and $p$ is subcritical. If $u_\gamma$ is a solution of \eqref{eq1.7}, then there exists a constant $C$, independent of $\gamma$, such that $\|u_\gamma\|_\infty\le C$ for all $0<\gamma<\gamma_0$.
\end{lemma}

\begin{proof}
We argue by contradiction. Assume there exists a sequence of solutions $u_n=u_{\gamma_n}$ such that $\|u_n\|_\infty\to\infty$ as $n\to\infty$. Define 
$$
\hat{u}_{n}(x)=\frac{1}{\|u_{n}\|_\infty}u_n\Big(\frac{x}{\|u_{n}\|_\infty^{\frac{p-1}{2}}}+x_{n}\Big),
$$
where $x_n$ is given by
$$
u_n(x_n)=\max_{x\in\Omega}u_n(x)=\|u_n\|_\infty.
$$
Then $\hat{u}_n$ solves the following equation 
\begin{equation*}
\begin{cases}
-\Delta \hat{u}_n=\hat{u}_n^{p}+\frac{\gamma}{\|u_n\|_\infty^{p-q}}\hat{u}_n^q,\quad &\text{in}\ \Omega_n,\\
\|u_n\|_\infty^{\frac{p_n-1}{2}}\frac{\partial \hat{u}_n}{\partial \nu}+\beta \hat{u}_n=0, &\text{on}\ \partial\Omega_n,
\end{cases}
\end{equation*}
where $\Omega_n=\|u_n\|^{\frac{p-1}{2}}(\Omega-x_n)$. Since 
$$
\frac{\gamma}{\|u_n\|_\infty^{p-q}}\to 0,\quad \text{as}\ n\to\infty,
$$
proceeding as in the proof of Lemma \ref{lemma2.4}, we deduce that 
$$
\hat{u}_n\to \hat{u}_\infty\quad \text{in}\  C^2_{loc}(\Omega_\infty),
$$ 
where $\hat{u}_\infty$ is a positive solution of 
\begin{equation}
\begin{cases}
-\Delta \hat{u}_\infty=\hat{u}_\infty^{p},\quad &\text{in}\ \Omega_\infty,\\
0\le\hat{u}_\infty\le 1 &\text{in}\ \Omega_\infty,\\
\frac{\hat{u}_\infty}{\partial x_N}=0, &\text{on}\ \partial\Omega_\infty \ (\text{if}\ \Omega_\infty=\R^N_+),
\end{cases}
\end{equation}
where $\Omega_\infty=\R^N$ or $\Omega_\infty=\R^N_+$. Moreover, we obtain that $\hat{u}_\infty\equiv 0$. This is a contradiction with $\hat{u}_\infty(0)=1$.
\end{proof}

{\bf Proof of Theorem \ref{theorem1.4}:} For any $\gamma<\gamma_0$, by Lemma \ref{lemma4.2}, equation \eqref{eq1.7} has a minimal solution $\bar{u}_\gamma$. We argue by contradiction. Suppose there exists a decreasing sequence $\gamma_n\to0$ such that the problem 
\begin{equation*}
\begin{cases}
-\Delta u=u^p+\gamma_n u^q,\quad &\text{in}\ \Omega,\\
u>0,  \quad &\text{in}\ \Omega,\\
\frac{\partial u}{\partial \nu}+\beta u=0, &\text{on}\ \partial\Omega,
\end{cases}
\end{equation*}
has two distinct solutions $u_n$ and $v_n$ other than the minimal solution $\bar{u}_{\gamma_n}$. By Lemma \ref{lemma4.4}, $u_n$ and $v_n$ are bounded. By elliptic regularity, there exist $u_0$ and $v_0$ such that $u_n\to u_0$ and $v_n\to v_0$ in $C^2(\bar{\Omega})$, where $u_0$ and $v_0$ solve 
\begin{equation}\label{m5}
\begin{cases}
-\Delta u=u^p,\quad &\text{in}\ \Omega,\\
u\ge0,  \quad &\text{in}\ \Omega,\\
\frac{\partial u}{\partial \nu}+\beta u=0, &\text{on}\ \partial\Omega.
\end{cases}
\end{equation}
This implies that either $u_0>0$ or $u_0=0$. If $u_0=0$, then $u_n\to 0$ in $C^2(\bar{\Omega})$. Thus, by Lemma \ref{lemma4.3}, we obtain $u_n=\bar{u}_{\gamma_n}$ for sufficiently large $n$, which is a contradiction since $u_n\ne \bar{u}_{\gamma_n}$. This shows that $u_0>0$. Similarly, we also have $v_0>0$. Utilizing Theorem \ref{theorem1.1}, we get 
$$
u_0=u=v_0,
$$
where $u$ is the unique positive solution of \eqref{eq0.1}. Define 
$$
w_n=u_n-v_n,
$$
then $\|w_n\|_\infty>0$ and $w_n$ is a solution of 
\begin{equation}
\begin{cases}
-\Delta w_n=b_n(x)w_n,\quad &\text{in}\ \Omega,\\
\frac{\partial w_n}{\partial\nu}+\beta w_n=0, &\text{on}\ \partial\Omega,
\end{cases}
\end{equation}
where $b_n(x)$ is defined by 
$$
b_n(x)=\int_0^1[p(\theta u_n+(1-\theta)v_n)^{p-1}+\gamma_nq(\theta u_n+(1-\theta)v_n)^{q-1}]d\theta.
$$
Let 
$$
\tilde{w}_n=\frac{w_n}{\|w_n\|_\infty},
$$
we have that $\tilde{w}_n$ solves 
\begin{equation*}
\begin{cases}
-\Delta \tilde{w}_n=b_n(x)\tilde{w}_n,\quad &\text{in}\ \Omega,\\
\frac{\partial \tilde{w}_n}{\partial\nu}+\beta \tilde{w}_n=0, &\text{on}\ \partial\Omega.
\end{cases}
\end{equation*}
Since $b_n \to p u^{p-1}$ as $n \to \infty$ and $|\tilde{w}_n|\le 1$, elliptic regularity implies that $\tilde{w}_n \to \tilde{w}_0$ strongly in $C^2(\bar{\Omega})$, where $\tilde{w}_0$ solves 
\begin{equation}
\begin{cases}
-\Delta \tilde{w}_0=pu^{p-1}\tilde{w}_0,\quad &\text{in}\  \Omega,\\
\frac{\partial \tilde{w}_0}{\partial\nu}+\beta\tilde{w}_0=0, &\text{on}\ \partial\Omega.
\end{cases}
\end{equation}
By the nondegeneracy of $u$, we obtain $\tilde{w}_0\equiv 0$. This is a contradiction since $\|\tilde{w}_0\|_\infty=1$. From the previous proof it also follows that,  as $\gamma \to 0$,  the minimal solution $\bar u_\gamma \to 0$ and the second solution $u_\gamma \to u_0$, where $u_0>0$ is a solution to \eqref{m5}.
\qed

\bibliographystyle{abbrv}
\bibliography{ChenGrossiLi1.bib}

\end{document}